\documentclass[11pt,a4paper]{article}

\usepackage[hmargin={25mm,25mm},vmargin={25mm,25mm}]{geometry}
\usepackage[utf8]{inputenc}
\usepackage{latexsym}
\usepackage{amssymb}
\usepackage{amsfonts}
\usepackage{mathtools}
\usepackage{commath}
\usepackage{bbm}
\usepackage{csquotes}
\usepackage{theoremref}
\usepackage{gensymb}
\usepackage{tikz}
\usepackage{authblk}

\usepackage{amsthm, mathtools, graphicx, paralist, cite, color, url, hyperref}
\usepackage{comment}

\usepackage{amsmath, nccmath}
\usepackage{indentfirst}
\usepackage{appendix}
\usepackage{listings}
\usepackage{hyphenat}
\usepackage{physics}
\usepackage{calrsfs}
\usepackage{enumerate}
\usepackage{tikz-cd}
\usepackage{mathrsfs}

\DeclareMathAlphabet{\mathcal}{OMS}{cmsy}{m}{n}
\SetMathAlphabet{\mathcal}{bold}{OMS}{cmsy}{b}{n}

\theoremstyle{plain}
\newtheorem{theorem}{Theorem}[section]
\newtheorem{lemma}[theorem]{Lemma}

\newtheorem{conj}[theorem]{Conjecture}
\newtheorem{obs}[theorem]{Observation}

\theoremstyle{definition}

\begin{document}

\title{Voronoi Games on the Discrete Hypercube: Four-Player Equilibria}
\author{Stelios Stylianou\footnote{School of Mathematics, University of Bristol. Supported by an EPSRC Doctoral Training Studentship.}}
\date{October 2025}

\maketitle

\begin{abstract}
    We consider a four-player game on the discrete hypercube $Q_n = \{0,1\}^n$, where each of the four players has chosen a single vertex of the hypercube. Such a position is called a \emph{profile}. Imagine there is a voter at every vertex, and each voter gives their vote to whichever player is closest to them, in terms of Hamming distance. If multiple players are tied for this smallest distance, the vote is divided equally between them. The \emph{score} of a player is the total number of votes they get. (This has a natural interpretation in terms of voting theory: imagine there are $n$ binary issues and that voters are uniformly distributed in their positions on these issues, and view the players as political candidates competing for vote share.) We say that a profile is an \emph{equilibrium} if no player can strictly increase their score by moving to a different vertex, while the other players maintain their original positions. Moreover, a profile is \emph{balanced} if, in each of the $n$ coordinates, two players have chosen 0, and two players have chosen 1. We prove that a four-player profile is an equilibrium if and only if it is balanced, proving a conjecture of Day and Johnson.
\end{abstract}

\section{Introduction}

The following `Voronoi game' on the discrete hypercube was introduced by Feldmann et al.\ \cite{three_player}. Let $n,k$ be positive integers, with $k \geq 2$, and let $Q_n \coloneqq \{0,1\}^n$ denote the (vertices of the) $n$-dimensional hypercube. Suppose we have $k$ players, say $A_1,...,A_k$, and each player $A_i$ has chosen some position $x^{(i)} \in Q_n$. Let $P=(x^{(1)},..., x^{(k)})$. We say that $P$ is a \emph{profile} on $Q_n$. For any $x,y \in Q_n$, let $d(x,y) \coloneqq \sum_{i=1}^n \abs{x_i-y_i}$ denote the Hamming distance between $x$ and $y$. We also write, for any $i \in [k]$, and any $x \in Q_n$, $d_P(x,A_i) \coloneqq d(x,x^{(i)})$, to make it clear that we are considering the player's position in profile $P$.

For any $x \in Q_n$, and any $i \in [k]$, we say that $A_i$ \emph{claims} $x$ (in $P$) if $d_P(x,A_i) \leq d_P(x,A_j)$, for all $j \in [k]$. It is possible that a tie occurs between various players for the smallest distance to a point $x$. We write $t_P(x)$ for the number of players that claim $x$ in $P$.

The \emph{Voronoi cell} of $A_i$ in $P$ is defined as $$V_P(A_i) \coloneqq \{x \in Q_n: A_i \ \text{claims} \ x\},$$ the \emph{ballot} of $x$ in $P$ is defined as 
$$\tau_{i,P}(x) \coloneqq \begin{cases*}
\frac{1}{t_P(x)} & for $x \in V_P(A_i)$, \\
0 & otherwise,
\end{cases*}$$
and the \emph{score} of $A_i$ in $P$ is defined as $$\sigma_P(A_i) \coloneqq 2^{-n} \sum_{x \in Q_n} \tau_{i,P}(x) = 2^{-n} \sum_{x \in V_P(A_i)} \frac{1}{t_P(x)}.$$

In other words, we can imagine that there is a voter at each vertex of $Q_n$, and that voter will give their vote to whichever player is closer to them in terms of Hamming distance. If there is a tie between multiple players, the vote will be equally divided among them.

We say that a profile $P$ is an \emph{equilibrium} if no player $A_i$ can (strictly) increase their score by moving from $x^{(i)}$ to some new position $x^{(i)'}$, while the other players stay at their initial positions. More formally, $P = (x^{(1)}, \dots, x^{(k)})$ is an equilibrium if, for all $i \in \{1, \dots, k\}$, $\sigma_P(A_i) \geq \sigma_{P'}(A_i)$, for any $P' = (x^{(1)}, \dots, x^{(i-1)}, x^{(i)'}, x^{(i+1)}, \dots, x^{(k)})$, where $x^{(i)'} \in Q_n$ is arbitrary. We are interested in determining, for any fixed integer $k \geq 2$, whether there exists an equilibrium with $k$ players on $Q_n$, for arbitrarily large $n$.

This problem was motivated from voting theory: we can imagine that each of the $n$ coordinates is associated with a `yes/no question', concerning some (binary) political issue, where answering `no' to a certain question corresponds to choosing $0$ in the respective coordinate, and answering `yes' corresponds to choosing $1$. Therefore, the position of a player/voter is uniquely determined once they answer all the questions. The distance between two individuals positioned on $Q_n$ is equal to the number of questions in which they disagree. In our model, we assume that the voters are distributed uniformly on the vertices of $Q_n$. This hypercube model is an example of a `spatial voting model', where the positions of candidates and voters on the political spectrum are represented using some geometric structure. See \cite{enelow1984spatial} and \cite{Unified} for discussion of prior work in this area.

Before looking at the hypercube game, we briefly mention some other examples of Voronoi games on graphs. A Voronoi game (named after Georgy Voronoi, who introduced Voronoi diagrams in 1908 \cite{Voronoi1908}) on a graph $G$ is typically played as follows: there is a finite number $k$ of players, and each of them picks a vertex of $G$. There is a voter at every vertex of $G$, and each voter gives their vote to whichever player is closest to them, in terms of graph distance. If there is a tie between multiple players for this smallest distance, the vote is divided equally between them. 

Mavronicolas et al.\ \cite{Voronoi:cycle} have shown that, if $G$ is the cycle graph $C_n$ on $n \geq k$ vertices, then an equilibrium (as defined above for the hypercube game) exists if and only if $k \leq 2n/3$ or $k=n$. Moreover, Feldmann et al.\ \cite{three_player} provide some general results about Voronoi games on transitive graphs, including some results on $Q_n$ that we mention below. Finally, Teramoto et al.\ \cite{Voronoi:tree} consider a two-player variant on the complete $m$-ary tree, which is played in multiple rounds, and each player chooses a single unoccupied vertex at each round.

We now focus on the hypercube problem. We observe that the question is trivial when $k = 2$. Indeed, for all values of $n$, any profile $P = (x^{(1)}, x^{(2)})$ is an equilibrium, as there exists an automorphism of $Q_n$ that swaps the points $x^{(1)}$ and $x^{(2)}$, showing that the two players always get a score equal to $1/2$, and thus neither of them can strictly increase their score by moving to a different point.

However, the problem already becomes interesting when $k = 3$. For a non-negative integer $x$, we write $0_x$ (or $1_x$) to represent a sequence of $x$ consecutive $0$s (or $1$s) in a binary vector. Feldmann et al.\ \cite{three_player} show that, when $n$ is odd, the $3$-player profile $P = (0_n, 0_{n-1}1, 1_n)$ is an equilibrium. Moreover, Day and Johnson \cite{unpublished} have shown that, for all even $n$, there is no $3$-player equilibrium on $Q_n$.

Feldmann et al.\ \cite{three_player} have also shown that, for all $n$, the profile $P = (0_n, 0_n, 1_n, 1_n)$ is a $4$-player equilibrium, and, for any odd $n$, and any $m \in \{1, \dots, n-1\}$, the profile $P = (0_n, 0_m1_{n-m}, 1_m0_{n-m}, 1_n)$ is also an equilibrium.

Note that, for any fixed $k$, it is easy to construct some `trivial' $k$-player equilibria on $Q_n$, for small $n$. For example, a $k$-player equilibrium on $Q_1$ occurs when $\lfloor k/2 \rfloor$ players choose $0$, and $\lceil k/2 \rceil$ players choose $1$. Moreover, when $k=2^n$, a $k$-player equilibrium on $Q_n$ occurs when there is exactly one player at every vertex. Day and Johnson \cite{unpublished} conjecture that, for any $k \geq 5$, there are no $k$-player profiles in equilibrium for large enough $n$, but this remains open for all $k$:

\begin{conj}
\label{conj:k>=5}
    For any fixed $k \geq 5$, there exists some $N$ such that, for any $n \geq N$, there are no $k$-player profiles in equilibrium on $Q_n$.
\end{conj}

Our goal is to give a complete characterization of $4$-player equilibrium profiles, by obtaining a (simple) necessary and sufficient condition for a $4$-player profile to be an equilibrium. When $k = 4$, we say that $P$ is \emph{balanced} if, for any $j$, $x^{(i)}_j=0$ for exactly two values of $i$ (and $x^{(i)}_j=1$ for the other two values of $i$). We prove the following, which was conjectured by Day and Johnson \cite{unpublished}:

\begin{theorem}
\label{thm:hyp_main}
    Let $n$ be a positive integer. A profile of four players on $Q_n$ is an equilibrium if and only if it is balanced.
\end{theorem}

A sketch of the proof of Theorem \ref{thm:hyp_main} is about to follow. We will first show that, in any balanced profile, all players get a score equal to $1/4$. Then we will prove that, starting from any balanced profile $P$, if a player, say $A_4$, moves to a different point, giving profile $P'$, their score will be equal to at most $1/4$. For the other direction, we will start from a non-balanced profile $P$, and show that any player, say $A_4$ again, performing a `balancing move', which means choosing the least popular value among the other three players in each coordinate, gets a score of at least $1/4$ in the new profile $P'$. If any player's score in $P$ is less than $1/4$ we are done. Otherwise, we show that at least one of the four players gets a score strictly larger than $1/4$ after performing their balancing move. 

For both directions, we will write down the players' positions in $P'$ in some general form, and then specify a set of three inequalities that must be satisfied in order for an arbitrary point $x^*$ to belong to $A_4$'s Voronoi cell. Each inequality will correspond to $A_4$ being closer to $x^*$ than $A_i$ is, for $i=1,2,3$. We will now briefly explain how we prove that non-balanced implies non-equilibrium. 

In any arbitrary profile, we can assume that at most one of $A_1,A_2$, and $A_3$ has chosen $1$ in any given coordinate. Therefore, $A_4$ chooses $1$ in every coordinate when performing his balancing move. There are four types of coordinates: in one of them, all three of $A_1,A_2$, and $A_3$ have chosen $0$. In the other three, exactly one of them has chosen $1$, and the different types are determined according to who that player is. We split the $n$ coordinates into four intervals, corresponding to the four types. In order to decide whether $x^*$ belongs to $A_4$'s Voronoi cell, we only need to know how many $1$s it contains in each interval, and the latter can be described using four variables. We write $\beta_1, \beta_2, \beta_3$, and $\beta_4$ for the number of $1$s that $x^*$ has in each interval. Therefore, we end up with three inequalities involving the $\beta_i$. If `$\geq$' holds in all three, then $x^*$ is claimed by $A_4$. The number of equalities decides what fraction of $x^*$'s vote goes to $A_4$. 

We then prove that, as the four variables take all possible values, the average score of $A_4$ (taking into account `how often' each quadruple of $\beta_i$'s occurs) is at least $1/4$. The most important observation, which makes the calculations much easier, is that, if the four intervals have lengths $y_1, \dots, y_4$, then the number of points that have $\alpha_i$ $1$s in the $i$th interval is constant when $\alpha_i \in \{\beta_i, y_i-\beta_i \}$ for all $i$. It is enough to show that, for any fixed $\beta_i$, $A_4$'s average score is at least $1/4$ as $(\alpha_1, \dots, \alpha_4)$ takes 16 different values, which occur by setting $\alpha_i$ equal to either $\beta_i$ or $y_i-\beta_i$. 

When proving that balanced implies equilibrium, we end up with a set of three equations in six variables rather than four, and we wish to show that $A_4$'s average score is at most $1/4$, using a very similar method to the one described above.

The remainder of this paper is structured as follows. In Section \ref{sec:bal_equil}, we prove that any balanced profile is an equilibrium, by first showing that all players get an equal score in a balanced profile. In Section \ref{sec:equil_bal}, we prove that any profile in equilibrium is balanced, completing the proof of Theorem \ref{thm:hyp_main}. In Section \ref{sec:conclusion} we conclude by discussing open problems.

\section{Proof that a balanced profile is an equilibrium.}
\label{sec:bal_equil}
In this section, we prove the backward direction of Theorem \ref{thm:hyp_main}, which we restate as the following.

\begin{theorem}
\label{thm1:main}
    If a profile of four players on $Q_n$ is balanced, then it is an equilibrium.
\end{theorem}

We first make some preliminary remarks and outline our strategy. It is easy to show that, in any balanced profile, each player gets a score equal to $1/4$. This simple observation will form the basis of the proof of Theorem \ref{thm:hyp_main}. In order to show that any balanced profile is an equilibrium, we will show that, if the position is balanced, then any move performed by a single player cannot result in their score increasing above $1/4$. In order to show that a non-balanced profile cannot be an equilibrium, we do not explicitly use Lemma \ref{lem:balanced_equal}. However, the motivation behind proving this also comes from Lemma \ref{lem:balanced_equal}, as we will show that there exists at least one player who can get a score higher than $1/4$ by moving to the position `closest to balanced', which means choosing, in each coordinate, the least popular value ($0$ or $1$) between the other three players.

\begin{lemma}
\label{lem:balanced_equal}
    If a profile $P$ of four players on $Q_n$ is balanced, then, for any $i$, $$\sigma_P(A_i) = 1/4.$$
\end{lemma}

\begin{proof}
    An arbitrary balanced profile can be written as $x^{(1)} = 0_n$, $x^{(2)} = 0_m1_{n-m}$, $x^{(3)} = 1_mu$, and $x^{(4)} = 1_m (1_{n-m}-u)$, where $u$ is any binary vector of length $n-m$. Let $h:Q_n \to Q_n$ be an automorphism given by $h(x) \coloneqq x+x^{(2)}$. Applying $h$ results in $A_1$ and $A_2$ swapping positions, as well as $A_3$ and $A_4$ also swapping positions. This implies that $A_1$ and $A_2$ get the same score, and so do $A_3$ and $A_4$. Repeating the argument (with the automorphism given by adding $x^{(3)}$ to everything) shows that $A_1$ and $A_3$ get the same score, and therefore all four players get the same score.
\end{proof}

In order to prove Theorem \ref{thm1:main}, it suffices to show that, starting from a balanced profile, any player changing their position will get a score of at most $1/4$. Let $y_1, y_2, y_3$ be non-negative integers such that $y_1 + y_2 + y_3 = n$. Without loss of generality, we can begin with a balanced profile $P$ given by $x^{(1)} = 0_{y_1}0_{y_2}1_{y_3}, x^{(2)} = 0_{y_1}1_{y_2}0_{y_3}, x^{(3)} = 1_{y_1}0_{y_2}0_{y_3}$, and $x^{(4)} = 1_n$. Our goal is to show that, if $A_4$ moves to any other position, he will get a score of at most $1/4$. Let $z_1, \dots, z_6$ be non-negative integers, such that $y_1 = z_1 + z_2, y_2 = z_3 + z_4$, and $y_3 = z_5 + z_6$. Let $x^{(4)'} = 1_{z_1}0_{z_2}1_{z_3}0_{z_4}1_{z_5}0_{z_6}$, and let $P' = (x^{(1)}, x^{(2)}, x^{(3)}, x^{(4)'})$. We want to prove that $\sigma_{P'}(A_4) \leq 1/4$.

Let $x^* \in Q_n$ be an arbitrary point, and let $\beta_1$ denote the number of 1s in the first $z_1$ coordinates of $x^*$, $\beta_2$ the number of 0s in the next $z_2$ coordinates, $\beta_3$ the number of 1s in the next $z_3$ coordinates, $\beta_4$ the number of 0s in the next $z_4$ coordinates, $\beta_5$ the number of 1s in the next $z_5$ coordinates, and $\beta_6$ the number of 0s in the last $z_6$ coordinates of $x^*$. In other words, $\beta_i$ shows the number of coordinates in which $x^*$ agrees with $x^{(4)'}$ in the interval corresponding to $z_i$.

We will write down a set of equations that must be satisfied in order for $x^*$ to belong to $V_{P'}(A_4)$. In order for this to happen, no other player can be closer to $x^*$ than $A_4$ is, so we must have $d_{P'}(x^*,A_4) \leq d_{P'}(x^*,A_i)$, for $i=1, 2, 3$. For a certain value of $i$, this is satisfied if and only if $x^*$ agrees with $x^{(4)'}$ in at least half of the coordinates in which $x^{(4)'}$ and $x^{(i)}$ disagree. The value taken by $x^*$ in the remaining coordinates is irrelevant. This gives the following set of equations:

\begin{equation}
\label{eq1:general}
    \begin{aligned}
        \beta_1+\beta_3+\beta_6 &\geq \frac{z_1+z_3+z_6}{2}, \\
        \beta_1+\beta_4+\beta_5 &\geq \frac{z_1+z_4+z_5}{2}, \\
        \beta_2+\beta_3+\beta_5 &\geq \frac{z_2+z_3+z_5}{2}.
    \end{aligned}
\end{equation}
For simplicity, we set $\delta_i = 2\beta_i - z_i$, and (\ref{eq1:general}) becomes:

\begin{equation}
\label{eq1:general_simple}
    \begin{aligned}
    \delta_1+\delta_3+\delta_6 &\geq 0, \\  
    \delta_1+\delta_4+\delta_5 &\geq 0, \\
    \delta_2+\delta_3+\delta_5 &\geq 0.
    \end{aligned}
\end{equation}

Let $\delta = (\delta_1, \dots, \delta_6)$. We say that a point $x^*$, as defined above, is a \emph{point of type $\delta$}. The number of points of type $\delta$ is equal to
\begin{equation*}
    \prod_{i=1}^6 \binom{z_i}{\beta_i} = \prod_{i=1}^6 \binom{z_i}{\frac{\delta_i + z_i}{2}} \coloneqq f(\delta).
\end{equation*}
We observe that $f$ is symmetric in each $\delta_i$. Therefore, if $\delta$ and $\delta'$ are chosen such that $\abs{\delta_i} = \abs{\delta_i'}$, for every $i$, then $f(\delta) = f(\delta')$.

For any $\delta$, we define the \emph{reward} $r(\delta)$ of $A_4$ from points of type $\delta$ as follows: if $\delta$ does not satisfy (\ref{eq1:general_simple}), then $r(\delta) = 0$. If $\delta$ satisfies (\ref{eq1:general_simple}), then 
$$r(\delta) = 
\begin{cases}
1, & \text{if strict inequality holds in all three equations in (\ref{eq1:general_simple})}, \\
1/2, & \text{if equality holds in exactly one equation in (\ref{eq1:general_simple})}, \\
1/3, & \text{if equality holds in exactly two equations in (\ref{eq1:general_simple})},\\
1/4, & \text{if equality holds in all three equations in (\ref{eq1:general_simple})}.
\end{cases}$$

For $i=1, \dots, 6$, let $\alpha_i$ be a non-negative integer, and let $\alpha = (\alpha_1, \dots, \alpha_6)$. Let $$S_\alpha \coloneqq \{\delta \in \mathbb{Z}^6 : \abs{\delta_i} = \alpha_i \ \text{for} \ i=1,\dots, 6\}.$$
It suffices to show that, for every $\alpha$, the average reward of $A_4$ from points of type $\delta$, where $\delta$ takes values in $S_\alpha$, is at most $1/4$.
Since $f$ is constant on $S_\alpha$, it is sufficient to prove that, for any $\alpha$,
\begin{equation}
    \label{eq1:alpha_good}
        \sum_{\delta \in S_\alpha} r(\delta) \leq \frac{\abs{S_\alpha}}{4}.
\end{equation}
We will say that $\alpha$ is \emph{good} if it satisfies equation (\ref{eq1:alpha_good}). We will first show that $\alpha$ is good if $\alpha_i > 0$ for all $i$:

\begin{lemma}
\label{lem1:main}
    If $\alpha_i > 0$, for $i=1, \dots, 6$, then $\alpha$ is good.
\end{lemma}

We begin by observing that, if some $\alpha_j$ is `too large' (compared to the other $\alpha_i$ that appear in the same equation), then it is easy to show that $\alpha$ is good. This will allow us to reduce the number of different cases that we need to check.

\begin{obs}
\label{obs1:good}
    Let $\alpha_1, \dots, \alpha_6$ be non-negative integers. If there exist distinct $i, j, k \in \{1, \dots, 6\}$, such that $\delta_i, \delta_j$, and $\delta_k$ appear in the same inequality in (\ref{eq1:general_simple}), and also $\alpha_i > \alpha_j + \alpha_k$, then $\alpha$ is good.
\end{obs}

\begin{proof}
    Suppose, without loss of generality, that $\alpha_i > \alpha_j + \alpha_k$, where $i,j,k$ are equal to $1,3,6$ in some order. Then, for any $\delta \in S_\alpha$, the sum $\delta_1 + \delta_3 + \delta_6$ is positive if and only if $\delta_i = \alpha_i$. We observe that either the second or third inequality in (\ref{eq1:general_simple}) does not involve $\delta_i$. Suppose the third inequality does not involve $\delta_i$. If $\delta_2 + \delta_3 + \delta_5 < 0$, then $r(\delta) = 0$, and if $\delta_2 + \delta_3 + \delta_5 = 0$, then $r(\delta) \leq 1/2$. Since $\delta_2 + \delta_3 + \delta_5$ is positive or negative with equal probability, and it is also independent of $\delta_i$, it follows that, among all $\delta \in S_\alpha$ that satisfy $\delta_i = \alpha_i$, $r(\delta) \leq 1/2$ on average. Since $r(\delta) = 0$ whenever $\delta_i = -\alpha_i$, it follows that $r(\delta) \leq 1/4$ on average.
\end{proof}

\begin{proof}[Proof of Lemma \ref{lem1:main}.]
Since $\alpha_i > 0$ for all $i$, $S_\alpha$ has $64$ distinct elements, so we need to show that $$\sum_{\delta \in S_\alpha} r(\delta) \leq 16.$$ We simply need to calculate this sum for the various cases that arise, depending on certain relations between the $\alpha_i$. Let $i,j,k$ be distinct, and $\delta_i, \delta_j, \delta_k$ appear in the same equation in (\ref{eq1:general_simple}). Using Observation \ref{obs1:good}, we can always assume that $\alpha_i + \alpha_j \geq \alpha_k$. We will say that $\alpha_k$ is \emph{big} in some equation if $\alpha_k = \alpha_i + \alpha_j$. For example, `$\alpha_5$ is big in the second equation' will mean that $\alpha_5 = \alpha_1 + \alpha_4$. We will consider four different cases, depending on how many of the equations in (\ref{eq1:general_simple}) can equal zero, as $\delta$ varies over $S_\alpha$. For any $\delta$, we define $$I_\delta^+ \coloneqq \{i \in \{1, \dots, 6\}:\delta_i > 0\},$$ and $$I_\delta^- \coloneqq \{i \in \{1, \dots, 6\}:\delta_i < 0\}.$$

\vspace{\baselineskip}

\textbf{Case 1: Equality in none of the equations.}\par\nopagebreak
Firstly, suppose that no equation can equal zero, which means that $\alpha_i + \alpha_j  > \alpha_k$, whenever $i,j,k$ are distinct and $\delta_i, \delta_j, \delta_k$ appear in the same equation. This will be the easiest case, as $r(\delta)$ can only be equal to $0$ or $1$ here. More precisely, $r(\delta) = 1$ if at least two of the three $\delta_i$ that appear in each equation are positive (that is, equal to $\alpha_i$), and $r(\delta) = 0$, otherwise. 

We can immediately see that this cannot be satisfied when at most two of the $\delta_i$ are positive. If exactly three of the $\delta_i$ are positive, then $r(\delta) = 1$ if and only if $\delta = (\alpha_1, -\alpha_2, \alpha_3, -\alpha_4, \alpha_5, -\alpha_6)$, since we need at least six positive $\delta_i$ in total (over the three equations in (\ref{eq1:general_simple})), and this is only possible when $\delta_i$ is positive for all odd $i$. If exactly four of the $\delta_i$ are positive, then $r(\delta) = 1$ if and only if the values of $i$ for which $\delta_i$ is negative are both even (which gives three cases, namely the pairs $\{2,4\}, \{2,6\}, \{4,6\}$), or one of them is odd and one of them is even, and the corresponding $\delta_i$ never appear in the same equation (which gives another three cases, namely the pairs $\{1,2\}, \{3,4\}, \{5,6\}$). Therefore, we get six values of $\delta$ with exactly four positive $\delta_i$, such that $r(\delta) = 1$. Finally, if at least five of the $\delta_i$ are positive, then $r(\delta) = 1$, which gives seven more values of $\delta$. In total, there are $14$ values of $\delta$ for which $r(\delta) = 1$, so we get $$\sum_{\delta \in S_\alpha} r(\delta) = 14 < 16.$$

\vspace{\baselineskip}

\textbf{Case 2: Equality in exactly one equation.}\par\nopagebreak
Due to symmetry, we can assume that equality can hold in the first equation. We need to consider two different subcases: either $\alpha_6 = \alpha_1 + \alpha_3$, or $\alpha_3 = \alpha_1 + \alpha_6$ (which gives the same result as $\alpha_1 = \alpha_3 + \alpha_6$, by symmetry). The results are summarized in Table \ref{table:one_equality}.

\textbf{Subcase 1: $\boldsymbol{\alpha_6 = \alpha_1 + \alpha_3}$.}\par\nopagebreak
In order for $r(\delta)$ to be positive, we need at least five positive $\delta_i$ in total over the three equations in (\ref{eq1:general_simple}) (at least one in the first equation, and at least two in each of the other two equations). Therefore, $r(\delta) = 0$ if $\delta_i$ is positive for at most two values of $i$. If exactly three of the $\delta_i$ are positive, then $r(\delta) = 1/2$, if $\delta_1, \delta_3, \delta_5$ are positive (equality holds in first equation), and $r(\delta) = 0$, otherwise. If exactly four of the $\delta_i$ are positive, then
$$r(\delta) = 
\begin{cases}
1, & \text{if } I_\delta^- = \{1,2\}, \{2,4\}, \{3,4\}, \\
1/2, & \text{if } I_\delta^- = \{1,3\}, \{2,6\}, \{4,6\}, \{5,6\}, \\
0, & \text{otherwise}.
\end{cases}$$
If at least five of the $\delta_i$ are positive, then $r(\delta) = 1$, unless $\delta_6 = - \alpha_6$, in which case $r(\delta) = 1/2$. Putting everything together, we get $$\sum_{\delta \in S_\alpha} r(\delta) = 12 < 16.$$

\vspace{\baselineskip}

\textbf{Subcase 2: $\boldsymbol{\alpha_3 = \alpha_1 + \alpha_6}$.}\par\nopagebreak
Once again, $r(\delta) = 0$, if $\abs{I_\delta^+} \leq 2$. If $\abs{I_\delta^+} = 3$, then
$$r(\delta) = 
\begin{cases}
1, & \text{if } I_\delta^+ = \{1,3,5\}, \\
1/2, & \text{if } I_\delta^+ = \{3,4,5\}, \\
0, & \text{otherwise}.
\end{cases}$$
If $\abs{I_\delta^+} = 4$, then 
$$r(\delta) = 
\begin{cases}
1, & \text{if } I_\delta^- = \{1,2\}, \{2,4\}, \{2,6\}, \{4,6\}, \{5,6\}, \\
1/2, & \text{if } I_\delta^- = \{1,6\}, \{3,4\}, \\
0, & \text{otherwise}.
\end{cases}$$
If at least five of the $\delta_i$ are positive, then $r(\delta) = 1$, unless $\delta_3 = - \alpha_3$, in which case $r(\delta) = 1/2$. Putting everything together, we get $$\sum_{\delta \in S_\alpha} r(\delta) = 14 < 16.$$

\begin{table}[ht]
\centering
    \begin{tabular}{|c|c|c|}
    \hline
    Subcase & 1st equation & $\sum_{\delta \in S_\alpha} r(\delta)$ \\ [1ex] 
    \hline\hline
    1 & $\alpha_6 = \alpha_1 + \alpha_3$ & $12$ \\ [0.5ex]
    \hline
    2 & $\alpha_3 = \alpha_1 + \alpha_6$ & $14$ \\ [1ex]
    \hline
    \end{tabular}
    \caption{Summary of subcases when $\alpha_i > 0$ for all $i$, and equality can only hold in the first equation in (\ref{eq1:general_simple}).}
    \label{table:one_equality}
\end{table}

\vspace{\baselineskip}

\textbf{Case 3: Equality in exactly two equations.}\par\nopagebreak
Due to symmetry, we can assume that equality can hold in the first two equations. Nine subcases arise here, but these can be reduced to only six distinct subcases by taking advantage of the symmetry between $\delta_3$ and $\delta_5$, as well as between $\delta_4$ and $\delta_6$. The results are summarized in Table \ref{table:two_equalities}.

\textbf{Subcase 1: $\boldsymbol{\alpha_1 = \alpha_3 + \alpha_6}$ and $\boldsymbol{\alpha_1 = \alpha_4 + \alpha_5}$.}\par\nopagebreak
If $\abs{I_\delta^+} \leq 2$, then $r(\delta) = 0$. If $\abs{I_\delta^+} = 3$, then
$$r(\delta) = 
\begin{cases}
1, & \text{if } I_\delta^+ = \{1,3,5\}, \\
1/2, & \text{if } I_\delta^+ = \{1,2,3\}, \{1,2,5\}, \\
0, & \text{otherwise}.
\end{cases}$$
If $\abs{I_\delta^+} = 4$, then 
$$r(\delta) = 
\begin{cases}
1, & \text{if } I_\delta^- = \{2,4\}, \{2,6\}, \{3,4\}, \{4,6\}, \{5,6\}, \\
1/2, & \text{if } I_\delta^- = \{3,6\}, \{4,5\}, \\
1/3, & \text{if } I_\delta^- = \{1,2\}, \\
0, & \text{otherwise}.
\end{cases}$$
If at least five of the $\delta_i$ are positive, then $r(\delta) = 1$, unless $\delta_1 = - \alpha_1$, in which case $r(\delta) = 1/3$. Putting everything together, we get $$\sum_{\delta \in S_\alpha} r(\delta) = 14\frac{2}{3} < 16.$$

\vspace{\baselineskip}

\textbf{Subcase 2: $\boldsymbol{\alpha_1 = \alpha_3 + \alpha_6}$ and $\boldsymbol{\alpha_4 = \alpha_1 + \alpha_5}$.}\par\nopagebreak
Note that this gives the same result as $\alpha_6 = \alpha_1 + \alpha_3$ and $\alpha_1 = \alpha_4 + \alpha_5$. If $\abs{I_\delta^+} \leq 2$, then $r(\delta) = 0$. If $\abs{I_\delta^+} = 3$, then
$$r(\delta) = 
\begin{cases}
1/2, & \text{if } I_\delta^+ = \{1,3,5\}, \\
1/3, & \text{if } I_\delta^+ = \{1,2,5\}, \\
0, & \text{otherwise}.
\end{cases}$$
If $\abs{I_\delta^+} = 4$, then 
$$r(\delta) = 
\begin{cases}
1, & \text{if } I_\delta^- = \{2,6\}, \{5,6\}, \\
1/2, & \text{if } I_\delta^- = \{1,2\}, \{2,4\}, \{3,4\}, \{3,6\}, \{4,6\}, \\
1/3, & \text{if } I_\delta^- = \{1,5\}, \\
0, & \text{otherwise}.
\end{cases}$$
If at least five of the $\delta_i$ are positive, then $r(\delta) = 1$, unless either $\delta_1 = - \alpha_1$ or $\delta_4 = -\alpha_4$, in which case $r(\delta) = 1/2$. Putting everything together, we get $$\sum_{\delta \in S_\alpha} r(\delta) = 11\frac{2}{3} < 16.$$

\vspace{\baselineskip}

\textbf{Subcase 3: $\boldsymbol{\alpha_1 = \alpha_3 + \alpha_6}$ and $\boldsymbol{\alpha_5 = \alpha_1 + \alpha_4}$.}\par\nopagebreak
Note that this gives the same result as $\alpha_3 = \alpha_1 + \alpha_6$ and $\alpha_1 = \alpha_4 + \alpha_5$. If $\abs{I_\delta^+} \leq 2$, then $r(\delta) = 0$. If $\abs{I_\delta^+} = 3$, then
$$r(\delta) = 
\begin{cases}
1, & \text{if } I_\delta^+ = \{1,3,5\}, \\
1/2, & \text{if } I_\delta^+ = \{1,2,5\}, \\
1/3, & \text{if } I_\delta^+ = \{3,5,6\}, \\
0, & \text{otherwise}.
\end{cases}$$
If $\abs{I_\delta^+} = 4$, then 
$$r(\delta) = 
\begin{cases}
1, & \text{if } I_\delta^- = \{2,4\}, \{2,6\}, \{3,4\}, \{4,6\}, \\
1/2, & \text{if } I_\delta^- = \{1,2\}, \{3,6\}, \{5,6\}, \\
1/3, & \text{if } I_\delta^- = \{1,4\}, \\
0, & \text{otherwise}.
\end{cases}$$
If at least five of the $\delta_i$ are positive, then $r(\delta) = 1$, unless either $\delta_1 = - \alpha_1$ or $\delta_5 = -\alpha_5$, in which case $r(\delta) = 1/2$. Putting everything together, we get $$\sum_{\delta \in S_\alpha} r(\delta) = 13\frac{2}{3} < 16.$$

\vspace{\baselineskip}

\textbf{Subcase 4: $\boldsymbol{\alpha_3 = \alpha_1 + \alpha_6}$ and $\boldsymbol{\alpha_5 = \alpha_1 + \alpha_4}$.}\par\nopagebreak
If $\abs{I_\delta^+} \leq 2$, then $r(\delta) = 0$, unless $I_\delta^+ = \{3,5\}$, in which case $r(\delta) = 1/3$. If $\abs{I_\delta^+} = 3$, then
$$r(\delta) = 
\begin{cases}
1, & \text{if } I_\delta^+ = \{1,3,5\}, \\
1/2, & \text{if } I_\delta^+ = \{3,4,5\}, \{3,5,6\} \\
1/3, & \text{if } I_\delta^+ = \{2,3,5\}, \\
0, & \text{otherwise}.
\end{cases}$$
If $\abs{I_\delta^+} = 4$, then 
$$r(\delta) = 
\begin{cases}
1, & \text{if } I_\delta^- = \{1,2\}, \{2,4\}, \{2,6\}, \{4,6\}, \\
1/2, & \text{if } I_\delta^- = \{1,4\}, \{1,6\}, \{3,4\}, \{5,6\}, \\
0, & \text{otherwise}.
\end{cases}$$
If at least five of the $\delta_i$ are positive, then $r(\delta) = 1$, unless either $\delta_3 = - \alpha_3$ or $\delta_5 = -\alpha_5$, in which case $r(\delta) = 1/2$. Putting everything together, we get $$\sum_{\delta \in S_\alpha} r(\delta) = 14\frac{2}{3} < 16.$$

\vspace{\baselineskip}

\textbf{Subcase 5: $\boldsymbol{\alpha_6 = \alpha_1 + \alpha_3}$ and $\boldsymbol{\alpha_5 = \alpha_1 + \alpha_4}$.}\par\nopagebreak
Note that this gives the same result as $\alpha_3 = \alpha_1 + \alpha_6$ and $\alpha_4 = \alpha_1 + \alpha_5$. If $\abs{I_\delta^+} \leq 2$, then $r(\delta) = 0$. If $\abs{I_\delta^+} = 3$, then
$$r(\delta) = 
\begin{cases}
1/2, & \text{if } I_\delta^+ = \{1,3,5\}, \{3,5,6\}, \\
1/3, & \text{if } I_\delta^+ = \{2,5,6\}, \\
0, & \text{otherwise}.
\end{cases}$$
If $\abs{I_\delta^+} = 4$, then 
$$r(\delta) = 
\begin{cases}
1, & \text{if } I_\delta^- = \{1,2\}, \{2,4\}, \{3,4\}, \\
1/2, & \text{if } I_\delta^- = \{1,3\}, \{1,4\}, \{2,6\}, \{4,6\}, \\
1/3, & \text{if } I_\delta^- = \{5,6\}, \\
0, & \text{otherwise}.
\end{cases}$$
If at least five of the $\delta_i$ are positive, then $r(\delta) = 1$, unless either $\delta_5 = - \alpha_5$ or $\delta_6 = -\alpha_6$, in which case $r(\delta) = 1/2$. Putting everything together, we get $$\sum_{\delta \in S_\alpha} r(\delta) = 12\frac{2}{3} < 16.$$

\vspace{\baselineskip}

\textbf{Subcase 6: $\boldsymbol{\alpha_6 = \alpha_1 + \alpha_3}$ and $\boldsymbol{\alpha_4 = \alpha_1 + \alpha_5}$.}\par\nopagebreak
If $\abs{I_\delta^+} \leq 2$, then $r(\delta) = 0$. If $\abs{I_\delta^+} = 3$, then
$$r(\delta) = 
\begin{cases}
1/3, & \text{if } I_\delta^+ = \{1,3,5\}, \\
0, & \text{otherwise}.
\end{cases}$$
If $\abs{I_\delta^+} = 4$, then 
$$r(\delta) = 
\begin{cases}
1, & \text{if } I_\delta^- = \{1,2\}, \\
1/2, & \text{if } I_\delta^- = \{1,3\}, \{1,5\}, \{2,4\}, \{2,6\}, \{3,4\}, \{5,6\}, \\
1/3, & \text{if } I_\delta^- = \{4,6\}, \\
0, & \text{otherwise}.
\end{cases}$$
If at least five of the $\delta_i$ are positive, then $r(\delta) = 1$, unless either $\delta_4 = - \alpha_4$ or $\delta_6 = -\alpha_6$, in which case $r(\delta) = 1/2$. Putting everything together, we get $$\sum_{\delta \in S_\alpha} r(\delta) = 10\frac{2}{3} < 16.$$

\begin{table}[ht]
\centering
    \begin{tabular}{|c|c|c|c|}
    \hline
    Subcase & 1st equation & 2nd equation & $\sum_{\delta \in S_\alpha} r(\delta)$ \\ [1ex] 
    \hline\hline
    1 & $\alpha_1 = \alpha_3 + \alpha_6$ & $\alpha_1 = \alpha_4 + \alpha_5$ & $14\frac{2}{3}$ \\ [1ex]
    \hline
    2 & $\alpha_1 = \alpha_3 + \alpha_6$ & $\alpha_4 = \alpha_1 + \alpha_5$ & $11\frac{2}{3}$ \\ [1ex]
    \hline
    3 & $\alpha_1 = \alpha_3 + \alpha_6$ & $\alpha_5 = \alpha_1 + \alpha_4$ & $13\frac{2}{3}$ \\ [1ex]
    \hline
    4 & $\alpha_3 = \alpha_1 + \alpha_6$ & $\alpha_5 = \alpha_1 + \alpha_4$ & $14\frac{2}{3}$ \\ [1ex]
    \hline
    5 & $\alpha_6 = \alpha_1 + \alpha_3$ & $\alpha_5 = \alpha_1 + \alpha_4$ & $12\frac{2}{3}$ \\ [1ex]
    \hline
    6 & $\alpha_6 = \alpha_1 + \alpha_3$ & $\alpha_4 = \alpha_1 + \alpha_5$ & $10\frac{2}{3}$ \\ [1ex]
    \hline
    \end{tabular}
    \caption{Summary of subcases when $\alpha_i > 0$ for all $i$, and equality can only hold in the first two equations in (\ref{eq1:general_simple}).}
    \label{table:two_equalities}
\end{table}

\vspace{\baselineskip}

\textbf{Case 4: Equality in all three equations.}\par\nopagebreak
There is a total of 27 different subcases here. These can be reduced to only seven distinct subcases, by taking advantage of the symmetry between $\delta_1$, $\delta_3$, and $\delta_5$, as well as between $\delta_2$, $\delta_4$, and $\delta_6$. However, one of these subcases would be given by $\alpha_1 = \alpha_3 + \alpha_6$, $\alpha_5 = \alpha_1 + \alpha_4$, and $\alpha_3 = \alpha_2 + \alpha_5$, which is not possible as it would imply $\alpha_1 > \alpha_3 > \alpha_5 > \alpha_1$. This subcase occurs two times, namely when $\alpha_i$ is big for some odd $i$ in every equation, and no $\alpha_i$ is big more than once. The results are summarized in Table \ref{table:three_equalities}.

\textbf{Subcase 1: $\boldsymbol{\alpha_1 = \alpha_3 + \alpha_6}$, $\boldsymbol{\alpha_1 = \alpha_4 + \alpha_5}$ and $\boldsymbol{\alpha_3 = \alpha_2 + \alpha_5}$.}\par\nopagebreak
This subcase occurs when there exists some odd $i$ such that $\alpha_i$ is big in two equations, and in the remaining equation $\alpha_j$ is big for some odd $j$. There are three choices for $i$, and then two choices for $j$, thus a total of six choices. If $\abs{I_\delta^+} \leq 2$, then $r(\delta) = 0$, unless $I_\delta^+ = \{1,3\}$, in which case $r(\delta) = 1/3$. If $\abs{I_\delta^+} = 3$, then
$$r(\delta) = 
\begin{cases}
1, & \text{if } I_\delta^+ = \{1,3,5\}, \\
1/2, & \text{if } I_\delta^+ = \{1,2,3\}, \{1,3,4\}, \\
1/3, & \text{if } I_\delta^+ = \{1,2,5\}, \{1,3,6\}, \\
0, & \text{otherwise}.
\end{cases}$$
If $\abs{I_\delta^+} = 4$, then 
$$r(\delta) = 
\begin{cases}
1, & \text{if } I_\delta^- = \{2,4\}, \{2,6\}, \{4,6\}, \{5,6\}, \\
1/2, & \text{if } I_\delta^- = \{2,5\}, \{3,4\}, \{4,5\}, \\
1/3, & \text{if } I_\delta^- = \{1,2\}, \{3,6\}, \\
0, & \text{otherwise}.
\end{cases}$$
If at least five of the $\delta_i$ are positive, then $r(\delta) = 1$, unless $\delta_1 = -\alpha_1$, in which case $r(\delta) = 1/3$, or $\delta_3 = -\alpha_3$, in which case $r(\delta) = 1/2$. Putting everything together, we get $$\sum_{\delta \in S_\alpha} r(\delta) = 15 < 16.$$

\vspace{\baselineskip}

\textbf{Subcase 2: $\boldsymbol{\alpha_1 = \alpha_3 + \alpha_6}$, $\boldsymbol{\alpha_1 = \alpha_4 + \alpha_5}$ and $\boldsymbol{\alpha_2 = \alpha_3 + \alpha_5}$.}\par\nopagebreak
This subcase occurs when there exists some odd $i$ such that $\alpha_i$ is big in two equations, and in the remaining equation $\alpha_j$ is big for some even $j$. There are three choices for $i$, and then only one choice for $j$, thus a total of three choices. If $\abs{I_\delta^+} \leq 2$, then $r(\delta) = 0$, unless $I_\delta^+ = \{1,2\}$, in which case $r(\delta) = 1/4$. If $\abs{I_\delta^+} = 3$, then
$$r(\delta) = 
\begin{cases}
1/2, & \text{if } I_\delta^+ = \{1,2,3\}, \{1,2,5\}, \{1,3,5\}, \\
1/3, & \text{if } I_\delta^+ = \{1,2,4\}, \{1,2,6\}, \\
0, & \text{otherwise}.
\end{cases}$$
If $\abs{I_\delta^+} = 4$, then 
$$r(\delta) = 
\begin{cases}
1, & \text{if } I_\delta^- = \{3,4\}, \{4,6\}, \{5,6\}, \\
1/2, & \text{if } I_\delta^- = \{2,4\}, \{2,6\}, \{3,5\}, \{3,6\}, \{4,5\}, \\
1/4, & \text{if } I_\delta^- = \{1,2\}, \\
0, & \text{otherwise}.
\end{cases}$$
If at least five of the $\delta_i$ are positive, then $r(\delta) = 1$, unless $\delta_1 = -\alpha_1$, in which case $r(\delta) = 1/3$, or $\delta_2 = -\alpha_2$, in which case $r(\delta) = 1/2$. Putting everything together, we get $$\sum_{\delta \in S_\alpha} r(\delta) = 14 < 16.$$

\vspace{\baselineskip}

\textbf{Subcase 3: $\boldsymbol{\alpha_6 = \alpha_1 + \alpha_3}$, $\boldsymbol{\alpha_1 = \alpha_4 + \alpha_5}$ and $\boldsymbol{\alpha_5 = \alpha_2 + \alpha_3}$.}\par\nopagebreak
This subcase occurs when there exists some even $i$ such that $\alpha_i$ is big in some equation, say in equation ($*$), $\alpha_j$ is big in one of the other two equations, where $j$ is the unique odd index that does not appear in equation ($*$), and in the remaining equation $\alpha_k$ is big for some odd $k \neq j$. There are three choices for $i$, then two choices for the equation in which $\alpha_j$ is big, and then only one choice for $k$, thus a total of six choices. If $\abs{I_\delta^+} \leq 2$, then $r(\delta) = 0$. If $\abs{I_\delta^+} = 3$, then
$$r(\delta) = 
\begin{cases}
1/2, & \text{if } I_\delta^+ = \{1,3,5\}, \{1,5,6\}, \\
1/4, & \text{if } I_\delta^+ = \{1,2,3\}, \{4,5,6\}, \\
0, & \text{otherwise}.
\end{cases}$$
If $\abs{I_\delta^+} = 4$, then 
$$r(\delta) = 
\begin{cases}
1, & \text{if } I_\delta^- = \{2,4\}, \{3,4\}, \\
1/2, & \text{if } I_\delta^- = \{1,2\}, \{2,3\}, \{2,6\}, \{4,6\}, \\
1/3, & \text{if } I_\delta^- = \{1,3\}, \{4,5\}, \{5,6\}, \\
0, & \text{otherwise}.
\end{cases}$$
If at least five of the $\delta_i$ are positive, then $r(\delta) = 1$, unless $\delta_j = -\alpha_j$, for some $j \in \{1,5,6\}$, in which case $r(\delta) = 1/2$. Putting everything together, we get $$\sum_{\delta \in S_\alpha} r(\delta) = 12 < 16.$$

\vspace{\baselineskip}

\textbf{Subcase 4: $\boldsymbol{\alpha_6 = \alpha_1 + \alpha_3}$, $\boldsymbol{\alpha_1 = \alpha_4 + \alpha_5}$ and $\boldsymbol{\alpha_3 = \alpha_2 + \alpha_5}$.}\par\nopagebreak
This subcase occurs when there exists some even $i$ such that $\alpha_i$ is big in some equation, say in equation ($*$), and the $\alpha_j$ and $\alpha_k$ that also appear in ($*$), where $j$ and $k$ are distinct and both odd, are big in one of the other two equations. There are three choices for $i$, and then we know which $\alpha_l$ is big in each of the other two equations, thus there is a total of three choices. If $\abs{I_\delta^+} \leq 2$, then $r(\delta) = 0$, unless $I_\delta^+ = \{1,3\}$, in which case $r(\delta) = 1/4$. If $\abs{I_\delta^+} = 3$, then
$$r(\delta) = 
\begin{cases}
1/2, & \text{if } I_\delta^+ = \{1,3,5\}, \\
1/3, & \text{if } I_\delta^+ = \{1,2,3\}, \{1,3,4\}, \{1,3,6\}, \\
0, & \text{otherwise}.
\end{cases}$$
If $\abs{I_\delta^+} = 4$, then 
$$r(\delta) = 
\begin{cases}
1, & \text{if } I_\delta^- = \{2,4\}, \\
1/2, & \text{if } I_\delta^- = \{1,2\}, \{2,5\}, \{2,6\}, \{3,4\}, \{4,5\}, \{4,6\}, \{5,6\}, \\
1/4, & \text{if } I_\delta^- = \{1,3\}, \\
0, & \text{otherwise}.
\end{cases}$$
If at least five of the $\delta_i$ are positive, then $r(\delta) = 1$, unless $\delta_j = -\alpha_j$, for some $j \in \{1,3,6\}$, in which case $r(\delta) = 1/2$. Putting everything together, we get $$\sum_{\delta \in S_\alpha} r(\delta) = 12 < 16.$$

\vspace{\baselineskip}

\textbf{Subcase 5: $\boldsymbol{\alpha_6 = \alpha_1 + \alpha_3}$, $\boldsymbol{\alpha_4 = \alpha_1 + \alpha_5}$ and $\boldsymbol{\alpha_3 = \alpha_2 + \alpha_5}$.}\par\nopagebreak
This subcase occurs when there exist some distinct even $i,j$, such that $\alpha_i$ and $\alpha_j$ are both big in their respective equations, and in the remaining equation $\alpha_k$ is big for some odd $k$. There are three choices for the pair $i,j$, and then two choices for $k$, thus a total of six choices. If $\abs{I_\delta^+} \leq 2$, then $r(\delta) = 0$. If $\abs{I_\delta^+} = 3$, then
$$r(\delta) = 
\begin{cases}
1/3, & \text{if } I_\delta^+ = \{1,3,4\}, \{1,3,5\}, \{3,4,6\}, \\
0, & \text{otherwise}.
\end{cases}$$
If $\abs{I_\delta^+} = 4$, then 
$$r(\delta) = 
\begin{cases}
1, & \text{if } I_\delta^- = \{1,2\}, \\
1/2, & \text{if } I_\delta^- = \{1,5\}, \{2,4\}, \{2,5\}, \{2,6\}, \{5,6\}, \\
1/3, & \text{if } I_\delta^- = \{1,3\}, \{3,4\}, \{4,6\}, \\
0, & \text{otherwise}.
\end{cases}$$
If at least five of the $\delta_i$ are positive, then $r(\delta) = 1$, unless $\delta_j = -\alpha_j$, for some $j \in \{3,4,6\}$, in which case $r(\delta) = 1/2$. Putting everything together, we get $$\sum_{\delta \in S_\alpha} r(\delta) = 11< 16.$$

\vspace{\baselineskip}

\textbf{Subcase 6: $\boldsymbol{\alpha_6 = \alpha_1 + \alpha_3}$, $\boldsymbol{\alpha_4 = \alpha_1 + \alpha_5}$ and $\boldsymbol{\alpha_2 = \alpha_3 + \alpha_5}$.}\par\nopagebreak
If $\abs{I_\delta^+} \leq 2$, then $r(\delta) = 0$. If $\abs{I_\delta^+} = 3$, then
$$r(\delta) = 
\begin{cases}
1/4, & \text{if } I_\delta^+ = \{1,3,5\}, \{2,4,6\}, \\
0, & \text{otherwise}.
\end{cases}$$
If $\abs{I_\delta^+} = 4$, then 
$$r(\delta) = 
\begin{cases}
1/2, & \text{if } I_\delta^- = \{1,2\}, \{1,3\}, \{1,5\}, \{3,4\}, \{3,5\}, \{5,6\}, \\
1/3, & \text{if } I_\delta^- = \{2,4\}, \{2,6\}, \{4,6\}, \\
0, & \text{otherwise}.
\end{cases}$$
If at least five of the $\delta_i$ are positive, then $r(\delta) = 1$, unless $\delta_j = -\alpha_j$, for some $j \in \{2,4,6\}$, in which case $r(\delta) = 1/2$. Putting everything together, we get $$\sum_{\delta \in S_\alpha} r(\delta) = 10 < 16.$$

\begin{table}[ht]
\centering
    \begin{tabular}{|c|c|c|c|c|}
    \hline
    Subcase & 1st equation & 2nd equation & 3rd equation & $\sum_{\delta \in S_\alpha} r(\delta)$ \\ [1ex] 
    \hline\hline
    1 & $\alpha_1 = \alpha_3 + \alpha_6$ & $\alpha_1 = \alpha_4 + \alpha_5$ & $\alpha_3 = \alpha_2 + \alpha_5$ & $15$ \\ [0.5ex]
    \hline
    2 & $\alpha_1 = \alpha_3 + \alpha_6$ & $\alpha_1 = \alpha_4 + \alpha_5$ & $\alpha_2 = \alpha_3 + \alpha_5$ & $14$ \\ [0.5ex]
    \hline
    3 & $\alpha_6 = \alpha_1 + \alpha_3$ & $\alpha_1 = \alpha_4 + \alpha_5$ & $\alpha_5 = \alpha_2 + \alpha_3$ & $12$ \\ [0.5ex]
    \hline
    4 & $\alpha_6 = \alpha_1 + \alpha_3$ & $\alpha_1 = \alpha_4 + \alpha_5$ & $\alpha_3 = \alpha_2 + \alpha_5$ & $12$ \\ [0.5ex]
    \hline
    5 & $\alpha_6 = \alpha_1 + \alpha_3$ & $\alpha_4 = \alpha_1 + \alpha_5$ & $\alpha_3 = \alpha_2 + \alpha_5$ & $11$ \\ [0.5ex]
    \hline
    6 & $\alpha_6 = \alpha_1 + \alpha_3$ & $\alpha_4 = \alpha_1 + \alpha_5$ & $\alpha_2 = \alpha_3 + \alpha_5$ & $10$ \\ [1ex]
    \hline
    \end{tabular}
    \caption{Summary of subcases when $\alpha_i > 0$ for all $i$, and equality can hold in all three equations in (\ref{eq1:general_simple}).}
    \label{table:three_equalities}
\end{table}
\end{proof}

\begin{proof}[Proof of Theorem \ref{thm1:main}.]

We already know, from Lemma \ref{lem1:main}, that $\alpha$ is good whenever $\alpha_i > 0$ for all $i$. Therefore, it is enough to prove that $\alpha$ is good when at least one of the $\alpha_i$ is equal to zero. 

Suppose there exist distinct $i,j$ such that $\delta_i$ and $\delta_j$ appear in the same equation in (\ref{eq1:general_simple}), and $\alpha_i = \alpha_j = 0$. Let $\delta_k$ also appear in that equation. If $\alpha_k > 0$, then we are done by Observation \ref{obs1:good}. So we can assume that $\alpha_k$ is also equal to zero. Without loss of generality, suppose that $\alpha_1 = \alpha_3 = \alpha_6 = 0$. If $\alpha_5 \neq \alpha_4$ or $\alpha_5 \neq \alpha_2$, then we are done by Observation \ref{obs1:good}. Suppose that $\alpha_2 = \alpha_4 = \alpha_5$. If they are all equal to $0$, then $\alpha$ is trivially good, since $r((0, \dots, 0)) = 1/4 = \abs{S_\alpha}/4$. If they are not equal to $0$, we get
$$r(\delta) = 
\begin{cases}
1/2, & \text{if } I_\delta^+ = \{2,4,5\}, \\
1/3, & \text{if } I_\delta^+ = \{2,5\}, \{4,5\}, \\
1/4, & \text{if } I_\delta^+ = \{2,4\}, \{5\}, \\
0, & \text{otherwise},
\end{cases}$$
which gives
$$\sum_{\delta \in S_\alpha} r(\delta) = 1\frac{2}{3} < 2 = \frac{\abs{S_\alpha}}{4}.$$

We can hereby assume that, if the distinct indices $i,j$ appear in the same equation in (\ref{eq1:general_simple}), then at most one of $\alpha_i$ and $\alpha_j$ is equal to $0$. This can only happen if at most three of the $\alpha_i$ are equal to $0$. If there are exactly three, they have to be $\alpha_2, \alpha_4$, and $\alpha_6$. In this case, we can assume that $\alpha_1 = \alpha_3 = \alpha_5$, since Observation \ref{obs1:good} implies that $\alpha$ is good in any other case. We have
$$r(\delta) = 
\begin{cases}
1, & \text{if } \abs{I_\delta^-} = 0, \\
1/3, & \text{if } \abs{I_\delta^-} = 1, \\
0, & \text{otherwise},
\end{cases}$$
which gives
$$\sum_{\delta \in S_\alpha} r(\delta) = 2 = \frac{\abs{S_\alpha}}{4}.$$

\vspace{\baselineskip}

Suppose now that exactly two of the $\alpha_i$ are equal to $0$. There are two cases to consider, up to symmetry. In the first case, $\alpha_2 = \alpha_4 = 0$. Using Observation \ref{obs1:good}, we can assume that $\alpha_1 = \alpha_3 = \alpha_5$, and also $\alpha_6 \leq \alpha_1 + \alpha_3$. If $\alpha_6 = \alpha_1 + \alpha_3$, then
$$r(\delta) = 
\begin{cases}
1, & \text{if } I_\delta^- = \emptyset, \\
1/2, & \text{if } I_\delta^- = \{1\}, \{3\}, \{6\}, \\
1/3, & \text{if } I_\delta^-= \{5\}, \\
1/4, & \text{if } I_\delta^- = \{1,3\}, \{5,6\}, \\
0, & \text{otherwise},
\end{cases}$$
which gives
$$\sum_{\delta \in S_\alpha} r(\delta) = 3\frac{1}{3} < 4 = \frac{\abs{S_\alpha}}{4},$$
and, if $\alpha_6 < \alpha_1 + \alpha_3$, then
$$r(\delta) = 
\begin{cases}
1, & \text{if } I_\delta^- = \emptyset, \{6\}, \\
1/2, & \text{if } I_\delta^- = \{1\}, \{3\}, \\
1/3, & \text{if } I_\delta^-= \{5\}, \{5,6\}, \\
0, & \text{otherwise},
\end{cases}$$
which gives
$$\sum_{\delta \in S_\alpha} r(\delta) = 3\frac{2}{3} < 4 = \frac{\abs{S_\alpha}}{4}.$$

\vspace{\baselineskip}

In the second case, $\alpha_1 = \alpha_2 = 0$. Using Observation \ref{obs1:good}, we can assume that $\alpha_3 = \alpha_6$, $\alpha_4 = \alpha_5$, and $\alpha_3 = \alpha_5$. Then
$$r(\delta) = 
\begin{cases}
1, & \text{if } I_\delta^- = \emptyset, \\
1/2, & \text{if } I_\delta^- = \{4\}, \{6\}, \\
1/3, & \text{if } I_\delta^-= \{3\}, \{5\}, \{4,6\}, \\
1/4, & \text{if } I_\delta^-= \{3,4\}, \{5,6\}, \\
0, & \text{otherwise},
\end{cases}$$
which gives
$$\sum_{\delta \in S_\alpha} r(\delta) = 3\frac{1}{2} < 4 = \frac{\abs{S_\alpha}}{4}.$$

\vspace{\baselineskip}

We finally check what happens when exactly one of the $\alpha_i$ is equal to $0$. In this case, $S_\alpha$ has $32$ distinct elements, so we need to show that $$\sum_{\delta \in S_\alpha} r(\delta) \leq 8.$$ Once again, we consider two different cases, up to symmetry. 

\textbf{Case 1: $\boldsymbol{\alpha_1=0}$.}\par\nopagebreak
Using Observation \ref{obs1:good}, we can assume that $\alpha_3 = \alpha_6$, and $\alpha_4 = \alpha_5$. We will consider three different subcases, up to symmetry. The results are summarized in Table \ref{table:a1=0}.

\textbf{Subcase 1: No equality in the third equation in (\ref{eq1:general_simple}).}\par\nopagebreak
If $\abs{I_\delta^+} \leq 1$, then $r(\delta) = 0$. If $\abs{I_\delta^+} = 2$, then 
$$r(\delta) = 
\begin{cases}
1/3, & \text{if } I_\delta^+ = \{3,5\}, \\
0, & \text{otherwise}.
\end{cases}$$
If $\abs{I_\delta^+} = 3$, then
$$r(\delta) = 
\begin{cases}
1/2, & \text{if } I_\delta^- = \{2,4\}, \{2,6\}, \\
1/3, & \text{if } I_\delta^- = \{3,4\}, \{4,6\}, \{5,6\}, \\
0, & \text{otherwise}.
\end{cases}$$
If $\abs{I_\delta^+} = 4$, then $r(\delta) = 1/2$, unless $\delta_2 = -\alpha_2$, in which case $r(\delta) = 1$. If $\abs{I_\delta^+} = 5$, then $r(\delta) = 1$. Putting everything together, we get $$\sum_{\delta \in S_\alpha} r(\delta) = 6\frac{1}{3} < 8.$$

\vspace{\baselineskip}

\textbf{Subcase 2: $\boldsymbol{\alpha_2 = \alpha_3 + \alpha_5}$.}\par\nopagebreak
If $\abs{I_\delta^+} \leq 1$, then $r(\delta) = 0$. If $\abs{I_\delta^+} = 2$, then 
$$r(\delta) = 
\begin{cases}
1/4, & \text{if } I_\delta^+ = \{3,5\}, \\
0, & \text{otherwise}.
\end{cases}$$
If $\abs{I_\delta^+} = 3$, then
$$r(\delta) = 
\begin{cases}
1/3, & \text{if } I_\delta^- = \{2,4\}, \{2,6\}, \{3,4\}, \{4,6\}, \{5,6\}, \\
1/4, & \text{if } I_\delta^- = \{3,5\}, \\
0, & \text{otherwise}.
\end{cases}$$
If $\abs{I_\delta^+} = 4$, then $r(\delta) = 1/2$. If $\abs{I_\delta^+} = 5$, then $r(\delta) = 1$. Putting everything together, we get $$\sum_{\delta \in S_\alpha} r(\delta) = 5\frac{2}{3} < 8.$$

\vspace{\baselineskip}

\textbf{Subcase 3: $\boldsymbol{\alpha_3 = \alpha_2 + \alpha_5}$.}\par\nopagebreak
Note that this gives the same result as $\alpha_5 = \alpha_2 + \alpha_3$. If $\abs{I_\delta^+} \leq 1$, then $r(\delta) = 0$. If $\abs{I_\delta^+} = 2$, then 
$$r(\delta) = 
\begin{cases}
1/3, & \text{if } I_\delta^+ = \{3,5\}, \\
1/4, & \text{if } I_\delta^+ = \{3,4\}, \\
0, & \text{otherwise}.
\end{cases}$$
If $\abs{I_\delta^+} = 3$, then
$$r(\delta) = 
\begin{cases}
1/2, & \text{if } I_\delta^- = \{2,4\}, \{2,6\}, \\
1/3, & \text{if } I_\delta^- = \{2,5\}, \{4,6\}, \{5,6\}, \\
1/4, & \text{if } I_\delta^- = \{3,4\}, \\
0, & \text{otherwise}.
\end{cases}$$
If $\abs{I_\delta^+} = 4$, then $r(\delta) = 1/2$, unless $I_\delta^- = \{2\}$, in which case $r(\delta) = 1$, or $I_\delta^- = \{3\}$, in which case $r(\delta) = 1/3$. If $\abs{I_\delta^+} = 5$, then $r(\delta) = 1$. Putting everything together, we get $$\sum_{\delta \in S_\alpha} r(\delta) = 6\frac{2}{3} < 8.$$

\begin{table}[ht]
\centering
    \begin{tabular}{|c|c|c|}
    \hline
    Subcase & 3rd equation in (\ref{eq1:general_simple}) & $\sum_{\delta \in S_\alpha} r(\delta)$ \\ [1ex] 
    \hline\hline
    1 & No equality & $6\frac{1}{3}$ \\ [1ex]
    \hline
    2 & $\alpha_2 = \alpha_3 + \alpha_5$ & $5\frac{2}{3}$ \\ [1ex]
    \hline
    3 & $\alpha_3 = \alpha_2 + \alpha_5$ & $6\frac{2}{3}$ \\ [1ex] 
    \hline
    \end{tabular}
    \caption{Summary of subcases when $\alpha_1 = 0$, and no other $\alpha_i$ is zero.}
    \label{table:a1=0}
\end{table}

\vspace{\baselineskip}

\textbf{Case 2: $\boldsymbol{\alpha_6=0}$.}\par\nopagebreak
Using Observation \ref{obs1:good}, we can assume that $\alpha_1 = \alpha_3$. We will consider nine different subcases, up to symmetry. The results are summarized in Table \ref{table:a6=0}. In all nine subcases, if $\abs{I_\delta^+} \leq 1$, then $r(\delta) = 0$, and if $\abs{I_\delta^+} = 5$, then $r(\delta) = 1$.

\textbf{Subcase 1: No equality in second or third equation in (\ref{eq1:general_simple}).}\par\nopagebreak
If $\abs{I_\delta^+} = 2$, then $r(\delta) = 0$. If $\abs{I_\delta^+} = 3$, then
$$r(\delta) = 
\begin{cases}
1, & \text{if } I_\delta^- = \{2,4\}, \\
1/2, & \text{if } I_\delta^- = \{1,2\}, \{3,4\}, \\
0, & \text{otherwise}.
\end{cases}$$
If $\abs{I_\delta^+} = 4$, then $r(\delta) = 1$, unless $I_\delta^- = \{1\}$ or $I_\delta^- = \{3\}$, in which case $r(\delta) = 1/2$. Putting everything together, we get $$\sum_{\delta \in S_\alpha} r(\delta) = 7 < 8.$$

\vspace{\baselineskip}

In the next three subcases, equality can hold in the second equation in (\ref{eq1:general_simple}), but not in the third equation. 

\textbf{Subcase 2: $\boldsymbol{\alpha_1 = \alpha_4 + \alpha_5}$.}\par\nopagebreak
If $\abs{I_\delta^+} = 2$, then $r(\delta) = 0$. If $\abs{I_\delta^+} = 3$, then
$$r(\delta) = 
\begin{cases}
1, & \text{if } I_\delta^- = \{2,4\}, \\
1/2, & \text{if } I_\delta^- = \{3,4\}, \{4,5\}, \\
1/3, & \text{if } I_\delta^- = \{1,2\}, \\
0, & \text{otherwise}.
\end{cases}$$
If $\abs{I_\delta^+} = 4$, then $r(\delta) = 1$, unless $I_\delta^- = \{1\}$, in which case $r(\delta) = 1/3$, or $I_\delta^- = \{3\}$, in which case $r(\delta) = 1/2$. Putting everything together, we get $$\sum_{\delta \in S_\alpha} r(\delta) = 7\frac{1}{6} < 8.$$

\vspace{\baselineskip}

\textbf{Subcase 3: $\boldsymbol{\alpha_4 = \alpha_1 + \alpha_5}$.}\par\nopagebreak
If $\abs{I_\delta^+} = 2$, then $r(\delta) = 0$. If $\abs{I_\delta^+} = 3$, then
$$r(\delta) = 
\begin{cases}
1/2, & \text{if } I_\delta^- = \{1,2\}, \{2,4\}, \\
1/3, & \text{if } I_\delta^- = \{1,5\}, \{3,4\}, \\
0, & \text{otherwise}.
\end{cases}$$
If $\abs{I_\delta^+} = 4$, then $r(\delta) = 1/2$, unless $I_\delta^- = \{2\}$ or $I_\delta^- = \{5\}$, in which case $r(\delta) = 1$. Putting everything together, we get $$\sum_{\delta \in S_\alpha} r(\delta) = 6\frac{1}{6} < 8.$$

\vspace{\baselineskip}

\textbf{Subcase 4: $\boldsymbol{\alpha_5 = \alpha_1 + \alpha_4}$.}\par\nopagebreak
If $\abs{I_\delta^+} = 2$, then 
$$r(\delta) = 
\begin{cases}
1/3, & \text{if } I_\delta^+ = \{3,5\}, \\
0, & \text{otherwise}.
\end{cases}$$
If $\abs{I_\delta^+} = 3$, then
$$r(\delta) = 
\begin{cases}
1, & \text{if } I_\delta^- = \{2,4\}, \\
1/2, & \text{if } I_\delta^- = \{1,2\}, \{3,4\}, \\
1/3, & \text{if } I_\delta^- = \{1,4\}, \\
0, & \text{otherwise}.
\end{cases}$$
If $\abs{I_\delta^+} = 4$, then $r(\delta) = 1/2$, unless $I_\delta^- = \{2\}$ or $I_\delta^- = \{4\}$, in which case $r(\delta) = 1$. Putting everything together, we get $$\sum_{\delta \in S_\alpha} r(\delta) = 7\frac{1}{6} < 8.$$

\vspace{\baselineskip}

In the next five subcases, equality can hold in both the second and third equation in (\ref{eq1:general_simple}). One other possible subcase would be given by $\alpha_1 = \alpha_4 + \alpha_5$ and $\alpha_5 = \alpha_2 + \alpha_3$, which is not possible as it would imply $\alpha_1 > \alpha_5 > \alpha_3 = \alpha_1$. Similarly, $\alpha_5 = \alpha_1 + \alpha_4$ and $\alpha_3 = \alpha_2 + \alpha_5$ cannot occur.

\textbf{Subcase 5: $\boldsymbol{\alpha_1 = \alpha_4 + \alpha_5}$ and $\boldsymbol{\alpha_3 = \alpha_2 + \alpha_5}$.}\par\nopagebreak
If $\abs{I_\delta^+} = 2$, then 
$$r(\delta) = 
\begin{cases}
1/3, & \text{if } I_\delta^+ = \{1,3\}, \\
0, & \text{otherwise}.
\end{cases}$$ 
If $\abs{I_\delta^+} = 3$, then
$$r(\delta) = 
\begin{cases}
1, & \text{if } I_\delta^- = \{2,4\}, \\
1/2, & \text{if } I_\delta^- = \{2,5\}, \{4,5\}, \\
1/3, & \text{if } I_\delta^- = \{1,2\}, \{3,4\}, \\
0, & \text{otherwise}.
\end{cases}$$
If $\abs{I_\delta^+} = 4$, then $r(\delta) = 1$, unless $I_\delta^- = \{1\}$ or $I_\delta^- = \{3\}$, in which case $r(\delta) = 1/3$. Putting everything together, we get $$\sum_{\delta \in S_\alpha} r(\delta) = 7\frac{2}{3} < 8.$$

\vspace{\baselineskip}

\textbf{Subcase 6: $\boldsymbol{\alpha_1 = \alpha_4 + \alpha_5}$ and $\boldsymbol{\alpha_2 = \alpha_3 + \alpha_5}$.}\par\nopagebreak
Note that this gives the same result as $\alpha_4 = \alpha_1 + \alpha_5$ and $\alpha_3 = \alpha_2 + \alpha_5$. If $\abs{I_\delta^+} = 2$, then 
$$r(\delta) = 
\begin{cases}
1/4, & \text{if } I_\delta^+ = \{1,2\}, \\
0, & \text{otherwise}.
\end{cases}$$ 
If $\abs{I_\delta^+} = 3$, then
$$r(\delta) = 
\begin{cases}
1/2, & \text{if } I_\delta^- = \{2,4\}, \{3,4\}, \{4,5\}, \\
1/3, & \text{if } I_\delta^- = \{3,5\}, \\
1/4, & \text{if } I_\delta^- = \{1,2\}, \\
0, & \text{otherwise}.
\end{cases}$$
If $\abs{I_\delta^+} = 4$, then $r(\delta) = 1$, unless $I_\delta^- = \{2\}$ or $I_\delta^- = \{3\}$, in which case $r(\delta) = 1/2$, or $I_\delta^- = \{1\}$, in which case $r(\delta) = 1/3$. Putting everything together, we get $$\sum_{\delta \in S_\alpha} r(\delta) = 6\frac{2}{3} < 8.$$

\vspace{\baselineskip}

\textbf{Subcase 7: $\boldsymbol{\alpha_4 = \alpha_1 + \alpha_5}$ and $\boldsymbol{\alpha_5 = \alpha_2 + \alpha_3}$.}\par\nopagebreak
Note that this gives the same result as $\alpha_5 = \alpha_1 + \alpha_4$ and $\alpha_2 = \alpha_3 + \alpha_5$. If $\abs{I_\delta^+} = 2$, then 
$$r(\delta) = 
\begin{cases}
1/4, & \text{if } I_\delta^+ = \{1,5\}, \\
0, & \text{otherwise}.
\end{cases}$$ 
If $\abs{I_\delta^+} = 3$, then
$$r(\delta) = 
\begin{cases}
1/2, & \text{if } I_\delta^- = \{1,2\}, \{2,4\}, \\
1/3, & \text{if } I_\delta^- = \{2,3\}, \{3,4\}, \\
1/4, & \text{if } I_\delta^- = \{1,5\}, \\
0, & \text{otherwise}.
\end{cases}$$
If $\abs{I_\delta^+} = 4$, then $r(\delta) = 1/2$, unless $I_\delta^- = \{2\}$, in which case $r(\delta) = 1$. Putting everything together, we get $$\sum_{\delta \in S_\alpha} r(\delta) = 6\frac{1}{6} < 8.$$

\vspace{\baselineskip}

\textbf{Subcase 8: $\boldsymbol{\alpha_4 = \alpha_1 + \alpha_5}$ and $\boldsymbol{\alpha_2 = \alpha_3 + \alpha_5}$.}\par\nopagebreak
If $\abs{I_\delta^+} = 2$, then 
$r(\delta) = 0$. If $\abs{I_\delta^+} = 3$, then
$$r(\delta) = 
\begin{cases}
1/3, & \text{if } I_\delta^- = \{1,2\}, \{1,5\}, \{2,4\}, \{3,4\}, \{3,5\}, \\
0, & \text{otherwise}.
\end{cases}$$
If $\abs{I_\delta^+} = 4$, then $r(\delta) = 1/2$, unless $I_\delta^- = \{5\}$, in which case $r(\delta) = 1$. Putting everything together, we get $$\sum_{\delta \in S_\alpha} r(\delta) = 5\frac{2}{3} < 8.$$

\vspace{\baselineskip}

\textbf{Subcase 9: $\boldsymbol{\alpha_5 = \alpha_1 + \alpha_4}$ and $\boldsymbol{\alpha_5 = \alpha_2 + \alpha_3}$.}\par\nopagebreak
If $\abs{I_\delta^+} = 2$, then 
$$r(\delta) = 
\begin{cases}
1/3, & \text{if } I_\delta^+ = \{1,5\}, \{3,5\}, \\
0, & \text{otherwise}.
\end{cases}$$ 
If $\abs{I_\delta^+} = 3$, then
$$r(\delta) = 
\begin{cases}
1, & \text{if } I_\delta^- = \{2,4\}, \\
1/2, & \text{if } I_\delta^- = \{1,2\}, \{3,4\}, \\
1/3, & \text{if } I_\delta^- = \{1,4\}, \{2,3\}, \\
0, & \text{otherwise}.
\end{cases}$$
If $\abs{I_\delta^+} = 4$, then $r(\delta) = 1$, unless $I_\delta^- = \{1\}$ or $I_\delta^- = \{3\}$, in which case $r(\delta) = 1/2$, or $I_\delta^- = \{5\}$, in which case $r(\delta) = 1/3$. Putting everything together, we get $$\sum_{\delta \in S_\alpha} r(\delta) = 7\frac{2}{3} < 8.$$
\end{proof}

\begin{table}[ht]
\centering
    \begin{tabular}{|c|c|c|c|}
    \hline
    Subcase & 2nd equation in (\ref{eq1:general_simple}) & 3rd equation in (\ref{eq1:general_simple}) & $\sum_{\delta \in S_\alpha} r(\delta)$ \\ [1ex] 
    \hline\hline
    1 & No equality & No equality & $7$ \\ [1ex]
    \hline
    2 & $\alpha_1 = \alpha_4 + \alpha_5$ & No equality & $7\frac{1}{6}$ \\ [1ex]
    \hline
    3 & $\alpha_4 = \alpha_1 + \alpha_5$ & No equality & $6\frac{1}{6}$ \\ [1ex] 
    \hline
    4 & $\alpha_5 = \alpha_1 + \alpha_4$ & No equality & $7\frac{1}{6}$ \\ [1ex] 
    \hline
    5 & $\alpha_1 = \alpha_4 + \alpha_5$ & $\alpha_3 = \alpha_2 + \alpha_5$ & $7\frac{2}{3}$ \\ [1ex] 
    \hline
    6 & $\alpha_1 = \alpha_4 + \alpha_5$ & $\alpha_2 = \alpha_3 + \alpha_5$ & $6\frac{2}{3}$ \\ [1ex] 
    \hline
    7 & $\alpha_4 = \alpha_1 + \alpha_5$ & $\alpha_5 = \alpha_2 + \alpha_3$ & $6\frac{1}{6}$ \\ [1ex] 
    \hline
    8 & $\alpha_4 = \alpha_1 + \alpha_5$ & $\alpha_2 = \alpha_3 + \alpha_5$ & $5\frac{2}{3}$ \\ [1ex] 
    \hline
    9 & $\alpha_5 = \alpha_1 + \alpha_4$ & $\alpha_5 = \alpha_2 + \alpha_3$ & $7\frac{2}{3}$ \\ [1ex] 
    \hline
    \end{tabular}
    \caption{Summary of subcases when $\alpha_6 = 0$, and no other $\alpha_i$ is zero.}
    \label{table:a6=0}
\end{table}

\section{Proof that a profile in equilibrium is balanced.}
\label{sec:equil_bal}
In this section, we prove the forward direction of Theorem \ref{thm:hyp_main}, which we restate as follows.
\begin{theorem}
\label{thm2:main}
    If a profile of four players on $Q_n$ is an equilibrium, then it is balanced.
\end{theorem}

We first outline our strategy. Let $P$ be an arbitrary profile on $Q_n$. We can assume, without loss of generality, that, in every coordinate, at least two of $A_1, A_2$, and $A_3$ have chosen 0. The players' positions can then be written as $x^{(1)} = 0_{y_1}0_{y_2}0_{y_3}1_{y_4}$, $x^{(2)} = 0_{y_1}0_{y_2}1_{y_3}0_{y_4}$, $x^{(3)} = 0_{y_1}1_{y_2}0_{y_3}0_{y_4}$, and $x^{(4)} = 1_{y_{1,1}}0_{y_{1,2}}1_{y_{2,1}}0_{y_{2,2}}1_{y_{3,1}}0_{y_{3,2}}1_{y_{4,1}}0_{y_{4,2}}$, where the non-negative integers $y_i$ and $y_{i,j}$ satisfy $y_1 + \dots + y_4 = n$, and $y_{i,1} + y_{i,2} = y_i$, for all $i$.

Our goal is to show that, if $P$ is not balanced, then at least one of the four players can increase their score by moving to the position that will make the profile `as close to balanced as possible'. More precisely, we say that a player performs a \emph{balancing move} if, for every $j$, that player chooses $x_j=0$ if and only if at least two of the other three players have chosen $x_j=1$. Otherwise, the player chooses $x_j=1$. Note that this will make the position balanced if and only if there is no coordinate in which the other three players all agree.

Let $P' = (x^{(1)}, x^{(2)}, x^{(3)}, 1_n)$, which is the profile that occurs when we start from $P$ and $A_4$ performs his balancing move. We will prove that $\sigma_{P'}(A_4) \geq 1/4$, with equality holding if and only if $P'$ is balanced. We will follow the same logic as in Section \ref{sec:bal_equil}.

Let $x^* \in Q_n$ be an arbitrary point, and let $\beta_1$ denote the number of 1s in the first $y_1$ coordinates of $x^*$, $\beta_2$ the number of 1s in the next $y_2$ coordinates, $\beta_3$ the number of 1s in the next $y_3$ coordinates, and $\beta_4$ the number of 1s in the last $y_4$ coordinates of $x^*$. In other words, $\beta_i$ shows the number of coordinates in which $x^*$ agrees with $x^{(4)'} = 1_n$ in the interval corresponding to $y_i$. We will first write down the set of equations that must be satisfied in order for $x^*$ to belong to $V_{P'}(A_4)$. We must have $d_{P'}(x^*,A_4) \leq d_{P'}(x^*,A_i)$, for $i=1, 2, 3$. Once again, we set $\delta_i = 2\beta_i - y_i$, which gives the following set of equations:

\begin{equation}
\label{eq2:general_simple}
    \begin{aligned}
    \delta_1+\delta_2+\delta_3 &\geq 0, \\  
    \delta_1+\delta_2+\delta_4 &\geq 0, \\
    \delta_1+\delta_3+\delta_4 &\geq 0.
    \end{aligned}
\end{equation}

Let $\delta = (\delta_1, \dots, \delta_4)$. We say that a point $x^*$, as defined above, is a \emph{point of type $\delta$}. The number of points of type $\delta$ is equal to
\begin{equation*}
    \prod_{i=1}^4 \binom{y_i}{\beta_i} = \prod_{i=1}^6 \binom{y_i}{\frac{\delta_i + y_i}{2}} \coloneqq g(\delta).
\end{equation*}
We observe that $g$ is symmetric in each $\delta_i$. Therefore, if $\delta$ and $\delta'$ are chosen such that $\abs{\delta_i} = \abs{\delta_i'}$, for every $i$, then $g(\delta) = g(\delta')$.

For any $\delta$, we define the \emph{reward} $\rho(\delta)$ of $A_4$ from points of type $\delta$ as follows: If $\delta$ does not satisfy (\ref{eq2:general_simple}), then $r(\delta) = 0$. If $\delta$ satisfies (\ref{eq2:general_simple}), then 
$$\rho(\delta) = 
\begin{cases}
1, & \text{if strict inequality holds in all three equations in (\ref{eq2:general_simple})}, \\
1/2, & \text{if equality holds in exactly one equation in (\ref{eq2:general_simple})}, \\
1/3, & \text{if equality holds in exactly two equations in (\ref{eq2:general_simple})},\\
1/4, & \text{if equality holds in all three equations in (\ref{eq2:general_simple})}.
\end{cases}$$

For $i=1, \dots, 4$, let $\alpha_i$ be a non-negative integer, and let $\alpha = (\alpha_1, \dots, \alpha_4)$. Let $$S_\alpha \coloneqq \{\delta \in \mathbb{Z}^4 : \abs{\delta_i} = \alpha_i \ \text{for} \ i=1,\dots, 4\}.$$
It suffices to show that, for every $\alpha$, the average reward of $A_4$ from points of type $\delta$, where $\delta$ takes values in $S_\alpha$, is at least $1/4$.
Since $g$ is constant on $S_\alpha$, it is sufficient to prove the following:

\begin{lemma}
\label{lem2:main}
    For any non-negative integers $\alpha_1, \dots, \alpha_4$,
\begin{equation}
    \label{eq2:alpha_good}
        \sum_{\delta \in S_\alpha} \rho(\delta) \geq \frac{\abs{S_\alpha}}{4}.
\end{equation}
\end{lemma}

\begin{proof}
     We say that $\alpha$ is \emph{good} if it satisfies (\ref{eq2:alpha_good}). For any $\delta \in S_\alpha$, let $I_\delta^+ = \{i \in \{1, \dots, 4\}: \delta_i > 0\}$, and $I_\delta^- = \{i \in \{1, \dots, 4\}: \delta_i < 0\}$. 
     
     \textbf{Case 1: $\boldsymbol{\alpha_i > 0}$, for all $\boldsymbol{i}$.}\par\nopagebreak
     Suppose that $\alpha_i \geq \alpha_j \geq \alpha_k \geq \alpha_l>0$, where $i,j,k,l$ are distinct, and equal to $1,2,3,4$ in some order. If $I_\delta^- = \emptyset, \{l\}$, or $\{k\}$, then $\rho(\delta) = 1$. If $i = 1$, then $\rho(\alpha_1, -\alpha_2, \alpha_3, \alpha_4) = 1$, so $\alpha$ is good. Otherwise, we can take, without loss of generality, $i=2$. Suppose first that $j=1$. If $\alpha_1 < \alpha_3 + \alpha_4$, then $\rho(-\alpha_1, \alpha_2, \alpha_3, \alpha_4) = 1$, if $\alpha_1 > \alpha_3 + \alpha_4$, then $\rho(\alpha_1, \alpha_2, -\alpha_3, -\alpha_4) = 1$, and if $\alpha_1 = \alpha_3 + \alpha_4$, then $\rho(\alpha_1, \alpha_2, -\alpha_3, -\alpha_4) =\rho(-\alpha_1, \alpha_2, \alpha_3, \alpha_4) = 1/2$. In all cases, $\alpha$ is good. Suppose now that $j=3$. If $\alpha_3 < \alpha_1 + \alpha_4$, then $\rho(\alpha_1,\alpha_2,-\alpha_3,\alpha_4) = 1$, and if $\alpha_3 > \alpha_1 + \alpha_4$, then $\rho(-\alpha_1,\alpha_2,\alpha_3,-\alpha_4) = 1$. If $\alpha_3 = \alpha_1 + \alpha_4$, then, if $\alpha_2 > \alpha_3$, then $\rho(-\alpha_1,\alpha_2,\alpha_3,-\alpha_4) = \rho(\alpha_1,\alpha_2,-\alpha_3,\alpha_4) = 1/2$, and if $\alpha_2 = \alpha_3$, then $\rho(\alpha_1,-\alpha_2,\alpha_3,\alpha_4) = \rho(\alpha_1,\alpha_2,-\alpha_3,\alpha_4) = 1/2$. In all cases, $\alpha$ is good.

     \vspace{\baselineskip}
     
     \textbf{Case 2: Exactly one $\boldsymbol{\alpha_i}$ is equal to $\boldsymbol{0}$.}\par\nopagebreak
     Suppose first that $\alpha_1 = 0$. If $\alpha_4 < \text{min}\{\alpha_2, \alpha_3\}$, then $\rho(\delta) = 1$ when $I_\delta^- = \emptyset, \{4\}$. If $\alpha_3 = \alpha_4 < \alpha_2$, then $\rho(\delta) = 1$ when $I_\delta^- = \emptyset$, and $\rho(\delta) = 1/2$ when $I_\delta^- = \{3\}, \{4\}$. If $\alpha_2 = \alpha_3 = \alpha_4$, then $\rho(\delta) = 1$ when $I_\delta^- = \emptyset$, and $\rho(\delta) = 1/3$ when $I_\delta^- = \{2\}, \{3\}, \{4\}$. In all cases, $\alpha$ is good. Suppose now that $\alpha_2 = 0$. If $\alpha_k < \text{min}\{\alpha_i, \alpha_j\}$, where $i,j,k$ is some permutation of $1,3,4$, then $\rho(\delta) = 1$ when $I_\delta^- = \emptyset, \{k\}$. If $\alpha_3 = \alpha_4 \leq \alpha_1$, then $\rho(\delta) = 1$ when $I_\delta^- = \emptyset$, and $\rho(\delta) \geq 1/2$ when $I_\delta^- = \{3\}, \{4\}$. If $\alpha_1 = \alpha_3 < \alpha_4$, then $\rho(\delta) = 1$ when $I_\delta^- = \emptyset$, and $\rho(\delta) = 1/2$ when $I_\delta^- = \{1\}, \{3\}$. Once again, $\alpha$ is good in all cases.

     \vspace{\baselineskip}
     
     \textbf{Case 3: At least two of the $\boldsymbol{\alpha_i}$ are equal to $\boldsymbol{0}$.}\par\nopagebreak
     If exactly two of them are $0$, then taking the other two of them to be positive gives a reward of $1$, and shows that $\alpha$ is good. If exactly three of the $\alpha_i$ are $0$, then taking the other one to be positive gives a reward of at least $1/2$, showing that $\alpha$ is good. Finally, if $\alpha_i = 0$ for all $i$, then $\rho(0,0,0,0) = 1/4$, and $\alpha$ is good.
\end{proof}

\begin{proof}[Proof of Theorem \ref{thm2:main}.]
Let $P$ be an arbitrary non-balanced profile. If any player's score is less than $1/4$, then, by Lemma \ref{lem2:main}, that player can increase their score by performing their balancing move. Suppose now that $\sigma_P(A_i) = 1/4$, for all $i$. Since $P$ is non-balanced, we can assume that there exists some coordinate in which at least three players agree. Without loss of generality, we can assume that $A_1, A_2$ and $A_3$ have all chosen $0$ in the first coordinate. The players' positions can then be written as $x^{(1)} = 0_{y_1}0_{y_2}0_{y_3}1_{y_4}$, $x^{(2)} = 0_{y_1}0_{y_2}1_{y_3}0_{y_4}$, $x^{(3)} = 0_{y_1}1_{y_2}0_{y_3}0_{y_4}$ and $x^{(4)} = 1_{y_{1,1}}0_{y_{1,2}}1_{y_{2,1}}0_{y_{2,2}}1_{y_{3,1}}0_{y_{3,2}}1_{y_{4,1}}0_{y_{4,2}}$, where $y_1 + \dots + y_4 = n$, $y_{i,1} + y_{i,2} = y_i$ for all $i$, and $y_1 > 0$. 

Let $P' = (x^{(1)}, x^{(2)}, x^{(3)}, 1_n)$ be the profile that occurs once $A_4$ performs his balancing move. We wish to prove that $\sigma_{P'}(A_4) > 1/4$. It suffices to show that there always exists at least one value of $\alpha$ for which strict inequality holds in Lemma \ref{lem2:main}. Recall that $\alpha_i$ has the same parity as $y_i$, for all $i$, so we check a few cases depending on the parity of the $y_i$. Suppose first that $y_1$ is odd. If $y_2, y_3, y_4$ are all even, then
\begin{equation*}
        \sum_{\delta \in S_{(1,0,0,0)}} \rho(\delta) = 1 > \frac{\abs{S_{(1,0,0,0)}}}{4} = \frac{1}{2}.
\end{equation*}
If $y_2$ is odd, and $y_3, y_4$ are even, then
\begin{equation*}
        \sum_{\delta \in S_{(1,1,0,0)}} \rho(\delta) = \frac{4}{3} > \frac{\abs{S_{(1,1,0,0)}}}{4} = 1.
\end{equation*}
If $y_2, y_3$ are odd, and $y_4$ is even, then
\begin{equation*}
        \sum_{\delta \in S_{(1,1,1,0)}} \rho(\delta) = \frac{7}{3} > \frac{\abs{S_{(1,1,1,0)}}}{4} = 2.
\end{equation*}
If $y_2, y_3, y_4$ are all odd, then
\begin{equation*}
        \sum_{\delta \in S_{(1,1,1,1)}} \rho(\delta) = 5 > \frac{\abs{S_{(1,1,1,1)}}}{4} = 4.
\end{equation*}

Suppose now that $y_1$ is even. Since $y_1 > 0$, we can always take $\alpha_1 = 2$. If $y_2, y_3, y_4$ are all even, then
\begin{equation*}
        \sum_{\delta \in S_{(2,0,0,0)}} \rho(\delta) = 1 > \frac{\abs{S_{(2,0,0,0)}}}{4} = \frac{1}{2}.
\end{equation*}
If $y_2$ is odd, and $y_3, y_4$ are even, then
\begin{equation*}
        \sum_{\delta \in S_{(2,1,0,0)}} \rho(\delta) = 2 > \frac{\abs{S_{(2,1,0,0)}}}{4} = 1.
\end{equation*}
If $y_2, y_3$ are odd, and $y_4$ is even, then
\begin{equation*}
        \sum_{\delta \in S_{(2,1,1,0)}} \rho(\delta) = \frac{7}{2} > \frac{\abs{S_{(2,1,1,0)}}}{4} = 2.
\end{equation*}
If $y_2, y_3, y_4$ are all odd, then
\begin{equation*}
        \sum_{\delta \in S_{(2,1,1,1)}} \rho(\delta) = 6 > \frac{\abs{S_{(2,1,1,1)}}}{4} = 4.
\end{equation*}
\end{proof}

\begin{proof}[Proof of Theorem \ref{thm:hyp_main}]
   Of course, Theorem \ref{thm:hyp_main} follows immediately from Theorems \ref{thm1:main} and \ref{thm2:main}.
\end{proof}

\section{Conclusion}
\label{sec:conclusion}

Having established Theorem \ref{thm:hyp_main}, we now understand equilibria in this Voronoi game with at most four players. The main open problem is to either prove or disprove Conjecture \ref{conj:k>=5}, i.e., to either prove that, for any fixed $k \geq 5$, no equilibria exist for large enough $n$, or to show that, for some value(s) of $k \geq 5$, equilibria exist for arbitrarily large $n$. Of course, if the latter is true, it would still be interesting to completely characterize the values of $k$ for which it happens. In general, answering the question for any specific $k \geq 5$ would be significant.

It is also worth mentioning that the same Voronoi game can be studied using a non-uniform measure on the vertices of $Q_n$. Once again, this is interesting from the voting theory perspective, as one would not expect the voters to be uniformly distributed among the $2^n$ possible positions, in a real-world scenario. Some two-player results in this direction have been established by Day and Johnson \cite{two-player}.

\section*{Acknowledgments}
The author would like to thank David Ellis for useful discussions, and Nick Day and Robert Johnson for sharing with us their manuscript \cite{unpublished}.

\bibliographystyle{plain}
\bibliography{references.bib}

\end{document}